%% file: complements.tex
\documentclass[a4paper,reqno]{article}

\input{header.tex}

\begin{document}
	
	\input{frontmatter.tex}
	
	\input{Introduction.tex}
	
	\input{Finite.tex}
	
	\input{WellSeparatingCommonComplements.tex}
	
	\input{Acknowledgments.tex}
	
	\input{backmatter.tex}
	
\end{document}

%% file: header.tex
\usepackage[utf8]{inputenc} 
\usepackage[T1]{fontenc} 
\usepackage{lmodern} 
\usepackage[american]{babel}
\usepackage{csquotes}

\usepackage{verbatim} 
\usepackage{authblk} 
\usepackage{calc}
\usepackage{amsmath}
\usepackage{amsfonts}
\usepackage{graphicx}
\usepackage{float}
\usepackage{xcolor}
\usepackage{enumitem}
\setlist[enumerate]{leftmargin=.5in} 
\setlist[itemize]{leftmargin=.5in} 
\providecommand{\keywords}[1]{\textbf{Keywords: } #1} 

\usepackage[
pdfmenubar=true,
pdftitle={Well-separating common complements in a Hilbert spaces are generic},
pdfauthor={Florian Noethen},
colorlinks=true,
linkcolor=black,
citecolor=blue,
filecolor=cyan,
urlcolor=magenta,
hidelinks=true,
raiselinks=false,
hyperfootnotes=true
]{hyperref}
\usepackage[capitalize,nameinlink]{cleveref} 

\def\natural{\mathbb{N}}

\def\real{\mathbb{R}}

\let\originalleft\left
\renewcommand{\left}{\mathopen{}\mathclose\bgroup\originalleft} 
\let\originalright\right\renewcommand{\right}{\aftergroup\egroup\originalright} 

\usepackage{amsthm}
\usepackage{amssymb}
\newtheorem*{theorem*}{Theorem} 
\newtheorem{theorem}{Theorem}[section] 
\newtheorem{lemma}[theorem]{Lemma} 
\newtheorem{proposition}[theorem]{Proposition}
\newtheorem{definition}[theorem]{Definition}

\newtheorem{remark}[theorem]{Remark}
\newtheorem{corollary}[theorem]{Corollary}
\newtheorem*{conjecture*}{Conjecture} 
\newtheorem{conjecture}[theorem]{Conjecture}
\Crefname{conjecture}{Conjecture}{Conjectures}

\usepackage{tikz}
\usetikzlibrary{arrows}
\usepackage{environ} 
\makeatletter
\newsavebox{\measure@tikzpicture}
\NewEnviron{scaletikzpicturetowidth}[1]{%
	\def\tikz@width{#1}%
	\begin{lrbox}{\measure@tikzpicture}%
		\BODY
	\end{lrbox}%
	\pgfmathparse{#1/\wd\measure@tikzpicture}%
	\BODY
}
\makeatother


%% file: frontmatter.tex
\title{Well-separating common complements of a sequence of subspaces of the same codimension in a Hilbert space are generic}
\author{Florian Noethen
	\thanks{Fachbereich Mathematik, Universität Hamburg, Bundesstraße 55, 20146 Hamburg, Germany (\href{mailto:florian.noethen@uni-hamburg.de}{florian.noethen@uni-hamburg.de}).}
}
\date{\today} 
\maketitle

\begin{abstract}
	Given a family of subspaces we investigate existence, quantity and quality of common complements in Hilbert spaces and Banach spaces. In particular we are interested in complements for countable families of closed subspaces of finite codimension. Those families naturally appear in the context of exponential type splittings like the multiplicative ergodic theorem, which recently has been proved in various infinite-dimensional settings. In view of these splittings, we show that common complements with subexponential decay of quality are generic in Hilbert spaces. Moreover, we prove that the existence of one such complement in a Banach space already implies that they are generic.
\end{abstract}


\hspace{1em}


\tableofcontents

\hypersetup{linkcolor=red}

%% file: Introduction.tex
\section{Introduction}\label{sectionIntroduction}
	
	Every proper subspace $V$ of a real vector space $X$ has a complementary subspace $C$, i.e. $X=V\oplus C$. Given two subspaces $V_1,V_2\subset X$ the existence of a common complement $C$ for both subspaces simultaneously becomes a more involved question \cite{drivaliaris2008subspaces,Drivaliaris2008,Lauzon2004}. A necessary requirement for the existence of a common complement is that $V_1$ and $V_2$ have the same codimension. In finite dimensions this requirement is enough to ensure the existence of a common complement. More generally, in \cite{Todd1990} it is shown that any countable family $(V_j)_{j\in\natural}$ of proper subspaces of the same codimension has uncountably many common complements if $\dim X<\infty$.\par
	Here, we are concerned with the infinite-dimensional setting, where $X$ is a Hilbert space or more generally a Banach space. In particular, we are seeking common complements for the class of closed subspaces of finite codimension. After briefly discussing the case of common complements for finitely many hyperplanes in \cref{sectionFiniteHyperplanes}, we will focus on common complements for countable families of subspaces in \cref{sectionWellSeparatingCommonComplements}. Central questions are the existence and quantity of common complements.\par 
	Just as important as the previous aspects is the quality of a complement. We will introduce a degree of transversality that indicates how close a complement is to stop being a complement or to being an optimal complement, which is the orthogonal complement in the Hilbert space setting. For a fixed complement $C$ the degree of transversality of each pairing $(V_j,C)$ should be as high as possible. While in general we cannot find a common lower bound on the degree of transversality, as the pairings may become worse with increasing index $j$, we will concentrate on the rate at which the quality may decrease. In view of exponential type splittings like the multiplicative ergodic theorem \cite{Gary2013SemiInvertibleMET,gonzalez-tokman_quas_2014,gonzalez2014concise}, we require that the quality decays at most subexponentially. Common complements fulfilling this requirement are called \emph{well-separating}. They open up new applications such as \cite{noethen2019METHilbert}.\par
	Using the concept of prevalence, we will prove the following theorem.
	\begin{theorem*}
		Well-separating common complements of a sequence of closed subspaces of the same codimension in a Hilbert space are generic.
	\end{theorem*}
	Since many techniques of the proof apply to Banach spaces, we will show that the existence of one well-separating common complement in a Banach space already implies that they are generic.

	\bigskip

%% file: Finite.tex
\section{Common complements for finitely many hyperplanes}\label{sectionFiniteHyperplanes}
	Before looking at countable families, this section deals with finite numbers of subspaces. Moreover, we restrict ourselves to the codimension $1$ case. Using simple geometric tools, we find common complements when $X=\real^n$ or when $X$ is an arbitrary Banach space. The quality of those complements motivates why we will require subexponential decay of the degree of transversality for families of subspaces in \cref{sectionWellSeparatingCommonComplements}.

	\begin{definition}\label{definitionGrassmannians}
		Let $X$ be a Banach space. The \emph{Grassmannian} $\mathcal{G}(X)$ is the set of closed complemented subspaces of $X$. It contains $\mathcal{G}_k(X)$, the set of $k$-dimensional subspaces, and $\mathcal{G}^k(X)$, the set of closed subspaces of codimension $k$. Elements of $\mathcal{G}^1(X)$ are called \emph{hyperplanes}.\par
		Moreover, we regard
		\begin{equation}\label{equationTransversality}
			\underset{x\in C,\ \|x\|=1}{\inf}\,\textnormal{d}(x,V)
		\end{equation}
		as the \emph{degree of transversality} between $V\in\mathcal{G}^k(X)$ and $C\in\mathcal{G}_k(X)$.
	\end{definition}

	Indeed, \cref{equationTransversality} takes values between zero and one. It is equal to zero if and only if $C$ is not a complement of $V$. If $X$ is a Hilbert space, then \cref{equationTransversality} equals one if and only if $C=V^{\perp}$.\\
	
	The next result gives us a geometric tool for finding common complements in $\real^n$.
	
	\begin{lemma}\label[lemma]{lemmaIntersectionPolytope}
		Let $P\subset\real^n$ be a compact, convex $n$-polytope with faces $(F_i)_{i=1}^m$ and normals $(f_i)_{i=1}^m$. Moreover, let $V\subset\real^n$ be a hyperplane with normal $v$. Then, the volume of the orthogonal projection of $P$ onto $V$ satisfies
		\begin{equation*}
		\textnormal{vol}_{n-1}(\Pi_V P)=\frac{1}{2}\sum_{i=1}^{m}\textnormal{vol}_{n-1}(F_i)|\langle f_i,v\rangle|.
		\end{equation*}
	\end{lemma}
	\begin{proof}
		This is a known result, see for example \cite{Burger1996}. The basic ideas are that $\Pi_V P=\Pi_V \partial P$ and that the interior of $\Pi_V \partial P$ is covered twice by the projection of the hull $\partial P$. Now, one only needs to check that $\textnormal{vol}_{n-1}(\Pi_V F_i)=\textnormal{vol}_{n-1}(F_i)|\langle f_i,v\rangle|$ for each face.
	\end{proof}
	
	\begin{corollary}\label[corollary]{corollaryIntersectionCube}
		Let $(V_j)_{j=1}^k$ be hyperplanes in $\real^n$. There exists $x\in\real^n$ with $\|x\|_2=1$ s.t. $\textnormal{d}(x,V_j)\geq \frac{1}{2}k^{-1}n^{-1}$.
	\end{corollary}
	\begin{proof}
		Let $P=[-1,1]^n$ and let $v_j$ be the normal of $V_j$. We have
		\begin{equation*}
		\textnormal{vol}_{n-1}(\Pi_{V_j} P)=2^{n-1}\sum_{i=1}^{n}|\langle e_i,v_j\rangle|=2^{n-1}\|v_j\|_1\leq 2^{n-1}\sqrt{n}.
		\end{equation*}
		Now, let $\delta:=\frac{1}{2}k^{-1}n^{-\frac{1}{2}}$. Denote by $\mu$ the Lebesgue measure on $\real^n$. We estimate
		\begin{align*}
		\mu\left(\left\{y\in P\ |\ \exists j:\ |\langle y,v_j\rangle|\leq\delta \right\}\right)&\leq\sum_{j=1}^{k}\mu\left(\left\{y\in P\ :\ |\langle y,v_j\rangle|\leq\delta \right\}\right)\\
		&\leq 2\delta\sum_{j=1}^{k}\textnormal{vol}_{n-1}(\Pi_{V_j} P)\\
		&\leq 2\delta k 2^{n-1}\sqrt{n}\\
		&= 2^{n-1}.
		\end{align*}
		Since $\textnormal{vol}(P)=2^n$, there must be an element $y\in P$ with $|\langle y,v_j\rangle|\geq\delta$ for all $j$. Writing $x:=y/\|y\|_2$ yields $\textnormal{d}(x,V_j)=|\langle x,v_j\rangle|\geq\delta/\|y\|_2\geq \frac{1}{2}k^{-1}n^{-1}$.
	\end{proof}
	
	\begin{remark}\label[remark]{remarkHyperplanesLowerBound}
		A lower bound better than $\frac{1}{2}k^{-1}n^{-1}$ for arbitrary hyperplanes is possible by looking at intersections of the unit ball and hyperplanes. In the case $V_j=\{x_j=0\}$ the best possible lower bound is $n^{-\frac{1}{2}}$.
	\end{remark}
	
	The next theorem is a well known result in the context of the Banach-Mazur compactum. As a consequence of John's theorem \cite{john2014} about ellipsoids, the maximal (multiplicative) distance of any Banach space of dimension $n$ to the standard euclidean space $\real^n$ is at most $\sqrt{n}$.
	
	\begin{theorem}\label[theorem]{theoremBanachMazur}
		Let $X$ be a Banach space of dimension $n$. There exists an isomorphism $T:(X,\|.\|)\to(\real^n,\|.\|_2)$ s.t. $\|T\|\|T^{-1}\|\leq\sqrt{n}$. 
	\end{theorem}
	
	\begin{corollary}\label[corollary]{corollaryBanachMazur}
		Let $(V_j)_{j=1}^k$ be hyperplanes of a Banach space $X$. There exists $x\in X$ of norm $1$ s.t. $\textnormal{d}(x,V_j)\geq \frac{1}{4}k^{-\frac{5}{2}}$.
	\end{corollary}
	\begin{proof}
		Set $V:=V_1\cap\dots\cap V_k$ and $Y:=X/V$. As a quotient space, $Y$ is a Banach space of dimension $n\leq k$. The quotient map $\pi:X\to Y$ sends $(V_j)_{j=1}^k$ to hyperplanes of $Y$. Now, by \cref{theoremBanachMazur} there is an isomorphism $T$ mapping $Y$ to $(\real^n,\|.\|_2)$ s.t. $\|T\|\leq\sqrt{n}$ and $\|T^{-1}\|\leq 1$. By \cref{corollaryIntersectionCube} we find $z\in\real^n$ with $\|z\|_2=1$ and $\textnormal{d}(z,T\pi V_j)\geq \frac{1}{2}k^{-1}n^{-1}$ for all $j$. Let $y:=T^{-1}z$. It holds $\|y\|_Y\leq 1$ and
		\begin{equation*}
		\frac{1}{2}k^{-1}n^{-1}\leq\underset{v_j\in V_j}{\inf}\|z-T\pi v_j\|_2\leq\underset{v_j\in V_j}{\inf}\|T\|\|y-\pi v_j\|_Y\leq\sqrt{n}\textnormal{d}(y,\pi V_j).
		\end{equation*}
		Take $x'\in X$ with $\pi x'=y$. Since $\inf_{v\in V}\|x'-v\|=\|y\|_Y\leq 1$, we find $v'\in V$ with $\|x'-v'\|\leq 2$. Set $x:=(x'-v')/\|x'-v'\|$. One readily checks that $\textnormal{d}(y,\pi V_j)=\textnormal{d}(x',V_j)=\|x'-v'\|\textnormal{d}(x,V_j)$. The claim follows.
	\end{proof}

	\begin{remark}\label[remark]{remarkExistenceArbitraryCodimFiniteNumber}
		From the proof of \cref{theoremMainHilbertSpaceExistence} for arbitrary codimensions together with \cref{corollaryBanachMazur} one can derive the existence of common complements for any finite number of subspaces $(V_j)_{j=1}^k\subset\mathcal{G}^k(X)$ of a Banach space.
	\end{remark}
	
	Given $k$ hyperplanes in a Banach space, \cref{corollaryBanachMazur} implies that there exists a common complement such that the minimal degree of transversality is bounded from below by $\frac{1}{4}k^{-\frac{5}{2}}$. On the other hand, \cref{remarkHyperplanesLowerBound} suggests that there are cases where $k^{-\frac{1}{2}}$ is the best that can be archived. Hence, as the number of hyperplanes is increased to infinity, we cannot hope for a common complement with positive minimal degree of transversality in general. Instead, we will ask for complements such that the degree of transversality decays at most subexponentially with the index of the hyperplane.

	\bigskip

%% file: WellSeparatingCommonComplements.tex
\section{Well-separating common complements}\label{sectionWellSeparatingCommonComplements}
	
	We briefly introduce our concept of well-separating common complements for families of subspaces. The existence of those complements is treated in \cref{subsectionExistence}. It will turn out that the existence of well-separating common complements for hyperplanes already implies the existence of well-separating common complements for subspaces of arbitrary codimension. In particular, they always exist if $X$ is a Hilbert space. Finally, \cref{subsectionPrevalence} will explain why the existence of one well-separating common complement in a Banach space $X$ is enough for them to be generic.
	
	\begin{definition}\label[definition]{definitionWellSeparatingComplement}
		Let $(X,\|.\|)$ be a Banach space, let $(V_j)_{j\in\natural}\subset\mathcal{G}^k(X)$, and let $\delta=(\delta_j)_{j\in\natural}\subset\real_{>0}$. A common complement $C\in\mathcal{G}_k(X)$ of $(V_j)_{j\in\natural}$ is called \emph{$\delta$-separating} if
		\begin{equation*}
		\forall j\in\natural\ :\hspace{1em}\underset{x\in C,\ \|x\|=1}{\inf}\textnormal{d}(x,V_j)\geq\delta_j.
		\end{equation*}
		Moreover, $C$ is called \emph{well-separating} if $C$ is $\delta$-separating for some $\delta$ with
		\begin{equation*}
		\lim_{j\to\infty}\frac{1}{j}\log \delta_j=0.
		\end{equation*}
	\end{definition}
	
	\begin{remark}\label[remark]{remarkWellSeparatingRational}
		A complement is well-separating if we can find $\delta$ with subexponential decay as $j\to\infty$. In particular, this holds true for polynomially decaying $\delta$.
	\end{remark}

\subsection{Existence}\label{subsectionExistence}

	\begin{theorem}\label[theorem]{theoremMainHilbertSpaceExistence}
		Let $H$ be a Hilbert space and let $(V_j)_{j\in\natural}\subset\mathcal{G}^k(H)$. There exists a well-separating common complement $C\in\mathcal{G}_k(H)$ of $(V_j)_{j\in\natural}$.
	\end{theorem}

	\begin{conjecture}\label[conjecture]{conjectureMainBanachSpaceExistence}
		Let $X$ be a Banach space and let $(V_j)_{j\in\natural}\subset\mathcal{G}^k(X)$. There exists a well-separating common complement $C\in\mathcal{G}_k(X)$ of $(V_j)_{j\in\natural}$.
	\end{conjecture}

	\cref{conjectureMainBanachSpaceExistence} is true if we can prove \cref{lemmaExistenceHilbertSpace} for Banach spaces. Furthermore, if $X$ is separable, then the problem reduces to solving \cref{lemmaExistenceHilbertSpace} for $X=l^1$. Indeed, every separable Banach space is isomorphic to a quotient of $l^1$. Now, let $\pi:l^1\to l^1/A$ be a quotient map. Then, $\pi$ induces a map $\mathcal{G}^1(l^1/A)\to\mathcal{G}^1(l^1)$ by $V\mapsto\pi^{-1}V$. It holds $\textnormal{d}(\alpha,\pi^{-1}V)=\textnormal{d}(\pi\alpha,V)$ for $\alpha\in l^1$. Hence, well-separating common complements in $l^1$ project onto well-separating common complements in $l^1/A$.\\
	
	We start with two lemmata needed to prove \cref{theoremMainHilbertSpaceExistence} for $k=1$. The first lemma is similar to \cref{corollaryIntersectionCube}.
	
	\begin{lemma}\label[lemma]{lemmaIntersectionBox}
		Let $(v_j)_{j=1}^n\subset\real^n$ be unit vectors s.t. $v_j\in\real^j\times\{0\}$. There exists an absolute constant $c>0$ and  $x=(x_1,\dots,x_n)^T\in\real^n$ s.t. $|x_j|\leq j^{-2}$ and $|\langle x,v_j\rangle|\geq cj^{-5}$.
	\end{lemma}
	\begin{proof}
		Let $P=\prod_{j=1}^{n}[-j^{-2},j^{-2}]$ and let $V_j$ be the hyperplane orthogonal to $v_j$. By \cref{lemmaIntersectionPolytope} we have
		\begin{align*}
		\textnormal{vol}_{n-1}(\Pi_{V_j} P)&=\sum_{i=1}^{n}\left(\prod_{k=1,k\neq i}^{n}2k^{-2}\right)|\langle e_i,v_j\rangle|\\
		&=\frac{1}{2}\textnormal{vol}_n(P)\sum_{i=1}^{n}i^2|\langle e_i,v_j\rangle|\\
		&=\frac{1}{2}\textnormal{vol}_n(P)\sum_{i=1}^{j}i^2|\langle e_i,v_j\rangle|\\
		&\leq \frac{1}{2}\textnormal{vol}_n(P)j^3.
		\end{align*}
		Now, let $\delta_j:=3\pi^{-2}j^{-5}$. We estimate
		\begin{align*}
		\mu\left(\left\{y\in P\ |\ \exists j:\ |\langle y,v_j\rangle|\leq\delta_j \right\}\right)&\leq\sum_{j=1}^{n}\mu\left(\left\{y\in P\ :\ |\langle y,v_j\rangle|\leq\delta_j \right\}\right)\\
		&\leq\sum_{j=1}^{n}2\delta_j\textnormal{vol}_{n-1}(\Pi_{V_j} P)\\
		&\leq\textnormal{vol}_n(P)\sum_{j=1}^{n}\delta_jj^3\\
		&\leq\textnormal{vol}_n(P)3\pi^{-2}\sum_{j=1}^{\infty}j^{-2}\\
		&=\frac{1}{2}\textnormal{vol}_n(P).
		\end{align*}
		Thus, there must be an element $y\in P$ with $|\langle y,v_j\rangle|\geq\delta_j$ for all $j$. Since $\|y\|_2^2\leq\sum_{j=1}^{\infty}j^{-4}=\frac{\pi^4}{90}$, writing $x:=y/\|y\|_2$ yields $|\langle x,v_j\rangle|\geq\delta_j/\|y\|_2\geq cj^{-5}$ with $c:=3\sqrt{90}\pi^{-4}$.
	\end{proof}
	
	\begin{lemma}\label[lemma]{lemmaExistenceHilbertSpace}
		Let $H$ be a Hilbert space and let $(\varphi_j)_{j\in\natural}\subset H'$ be a sequence of bounded linear functionals of norm $1$. There exists a sequence $(\delta_j)_{j\in\natural}\subset\real_{>0}$ and a unit vector $x\in H$ s.t.
		\begin{equation*}
		\lim_{j\to\infty}\frac{1}{j}\log \delta_j=0
		\end{equation*}
		and
		\begin{equation*}
		|\varphi_j(x)|\geq \delta_j
		\end{equation*}
		for all $j$.
	\end{lemma}
	\begin{proof}
		The case $\dim H<\infty$ follows from \cref{propositionPrevalenceFiniteDim} by observing that $|\varphi_j(x)|=\textnormal{d}(x,V_j)$ for $V_j:=\ker\varphi_j$. So, assume $\dim H = \infty$. By Riesz's representation theorem, we can write $\varphi_j=\langle v_j,\cdot\rangle$ for unit vectors $v_j\in H$. Now, take an orthonormal set $(c_j)_{j\in\natural}\subset H$ with $v_j\in\textnormal{span}(c_1,\dots,c_j)$. We get maps $\textnormal{pr}_j:H\to\real^j$ defined through 
		\begin{equation*}
		\textnormal{pr}_j(x):=\begin{pmatrix}
		\langle x,c_1\rangle\\
		\vdots\\
		\langle x,c_j\rangle
		\end{pmatrix}.
		\end{equation*}
		By construction $(\textnormal{pr}_j(v_i))_{i=1}^j\subset\real^j$ are unit vectors s.t. $\textnormal{pr}_j(v_i)\in\real^i\times\{0\}$. In particular, \cref{lemmaIntersectionBox} gives us the existence of an element $\alpha^j\in\prod_{k=1}^{j}[-k^{-2},k^{-2}]$ with $|\langle \alpha^j,\textnormal{pr}_j(v_i)\rangle|\geq ci^{-5}=:\tilde{\delta}_i$. Let $A_j$ be the set of all such $\alpha^j$, i.e.
		\begin{equation*}
		A_j:=\left\{\alpha^j\in\prod_{k=1}^{j}[-k^{-2},k^{-2}]\ \Bigg|\ \forall i\leq j:\ \left|\sum_{k=1}^j\alpha^j_k\langle v_i,c_k\rangle\right|\geq \tilde{\delta}_i\right\}.
		\end{equation*}
		We have shown that $A_j$ is a nonempty closed subset of $\real^j$. For $\alpha^j\in A_j$, we can define $y:=\sum_{k=1}^{j}\alpha_k^j c_k$. Since $\|y\|^2\leq\sum_{k=1}^{\infty}k^{-4}=\frac{\pi^4}{90}$, it holds
		\begin{equation*}
		|\varphi_i(x)|=\|y\|^{-1}|\langle v_i,y\rangle|=\|y\|^{-1}\left|\sum_{k=1}^{j}\alpha_k^j\langle v_i,c_k\rangle\right|\geq\sqrt{90}\pi^{-2}\tilde{\delta}_i=: \delta_i
		\end{equation*}
		for $i\leq j$, where $x:=y/\|y\|$. Thus, every $\alpha^j\in A_j$ yields an element $x\in H$ fulfilling the claim for $\varphi_1,\dots,\varphi_j$. The remainder of this proof treats the transition $j\to\infty$.\par
		By Tychonoff's theorem the space $\mathcal{B}:=\prod_{k=1}^{\infty}[-k^{-2},k^{-2}]$ equipped with the product topology is compact. Since the product topology is the coarsest topology such that the canonical projections $\pi_k:\mathcal{B}\to [-k^{-2},k^{-2}]$ are continuous, we find that $B_j:=(\pi_1\times\dots\times\pi_j)^{-1}A_j$ is a nonempty closed subset of $\mathcal{B}$. The set $B_j$ can be written as
		\begin{equation*}
		B_j=\left\{\alpha\in\mathcal{B}\ \Bigg|\ \forall i\leq j:\ \left|\sum_{k=1}^{\infty}\alpha_k\langle v_i,c_k\rangle\right|\geq \tilde{\delta}_i\right\}.
		\end{equation*}
		From this form is becomes obvious that $B_1\supset B_2\supset \dots$ is a decreasing sequence of closed subsets of $\mathcal{B}$. In particular, $(B_j)_{j\in\natural}$ has the finite intersection property, i.e. finite intersection are nonempty. As $\mathcal{B}$ is compact, the intersection of all $B_j$ must be nonempty. Thus, we find $\alpha$ in
		\begin{equation*}
		\bigcap_{k=1}^{\infty}B_j=\left\{\alpha\in\mathcal{B}\ \Bigg|\ \forall j\in\natural:\ \left|\sum_{k=1}^{\infty}\alpha_k\langle v_j,c_k\rangle\right|\geq \tilde{\delta}_j\right\}.
		\end{equation*}
		Similar to above, we set $y:=\sum_{k=1}^{\infty}\alpha_kc_k$. Again, it holds $\|y\|^2\leq\frac{\pi^4}{90}$. Defining $x:=y/\|y\|$, we get
		\begin{equation*}
		|\varphi_j(x)|=\|y\|^{-1}|\langle v_j,y\rangle|=\|y\|^{-1}\left|\sum_{k=1}^{\infty}\alpha_k\langle v_j,c_k\rangle\right|\geq\sqrt{90}\pi^{-2}\tilde{\delta}_j=\delta_j.
		\end{equation*}
	\end{proof}
	
	\begin{remark}\label[remark]{remarkExistenceHilbertSpaceDelta}
		The proof shows that $\delta_j$ can be chosen as $cj^{-5}$ for some constant $c>0$. Improvements of the exponent of $j$ are possible. For instance, one may use $j^{-(1+\epsilon)}$ instead of $j^{-2}$ to define the polytope in $\cref{lemmaIntersectionBox}$. However, since our goal is only to find an at most polynomially decaying lower bound, we aimed for better readability at the cost of a worse estimate. 
	\end{remark}
	
	So far it is not known to the author if \cref{lemmaExistenceHilbertSpace} holds for Banach spaces instead of Hilbert spaces. However, aside from \cref{lemmaExistenceHilbertSpace}, the remainder of this paper can be proved for Banach spaces. Hence, from now on assume that $(X,\|.\|)$ is an arbitrary Banach space.
	
	\begin{proof}[proof of \cref{theoremMainHilbertSpaceExistence} for $k=1$]
		By Hahn-Banach there are bounded linear functionals $(\varphi_j)_{j\in\natural}\subset X'$ of norm $1$ s.t. $\ker \varphi_j=V_j$. Assume that we find $(\delta_j)_{j\in\natural}$ and $x\in X$ as described by \cref{lemmaExistenceHilbertSpace}. Since $|\varphi_j(x)|=\textnormal{d}(x,V_j)$, the subspace spanned by $x$ is a $\delta$-well-separating common complement of $(V_j)_{j\in\natural}$.
	\end{proof}
	
	To prove \cref{theoremMainHilbertSpaceExistence} for arbitrary $k$ we need the following lemma.
	
	\begin{lemma}\label[lemma]{lemmaLine}
		Let $X$ be a Banach space. Furthermore, assume $x_1,x_2\in B(0,1)$ are two vectors s.t. $\|x_1\|\geq \mu_1$ and $\textnormal{d}(x_2,\textnormal{span}(x_1))\geq\mu_2$ for some numbers $0<\mu_1, \mu_2\leq 1$. Then, it holds
		\begin{equation*}
		\underset{t\in\real}{\inf}\|t x_1+(1-t)x_2\|\geq \frac{1}{2\sqrt{5}}\mu_1\mu_2.
		\end{equation*}
	\end{lemma}
	\begin{proof}
		The argument can be restricted to $\textnormal{span}(x_1,x_2)$. Thus, assume that $\dim X=2$. First, we look at $X=\real^2$ equipped with $\|.\|_2$. After a rotation we may assume $x_1=(\alpha_1,0)$ with $\alpha_1\geq\mu_1$. Now, the assumption on $x_2$ implies that its second coordinate has at least size $\mu_2$. Let $L$ be the line passing through $x_1$ and $x_2$ (see \cref{figurePlane}). We want to estimate the minimum distance between $L$ and the origin. Clearly, the distance becomes smallest if $L$ intersects the unit-circle at $(-\sqrt{1-\mu_2^2},\pm\mu_2)$. Hence, the task reduces to finding $\delta$ in \cref{figureTriangle}. Applying Pythagoras' theorem to find the diagonal $d$ of the big triangle and comparing ratios between catheti opposite to $\alpha$ and the hypotenuses, we find that
		\begin{equation*}
		\delta=\frac{\mu_1\mu_2}{d}=\frac{\mu_1\mu_2}{\sqrt{\mu_2^2+\left(\sqrt{1-\mu_2^2}+\mu_1\right)^2}}\geq\frac{1}{\sqrt{5}}\mu_1\mu_2.
		\end{equation*}
		Thus, the claim holds for the euclidean case.\par
		Now, let $X$ be any $2$-dimensional Banach space. By \cref{theoremBanachMazur} there exists an isomorphism $T$ from $X$ to $(\real^2,\|.\|_2)$ with $\|T\|\leq 1$ and $\|T^{-1}\|\leq\sqrt{2}$. Let $x_1,x_2\in X$ be as in the claim. It holds $Tx_1,Tx_2\in B(0,1)$, $\|Tx_1\|\geq \frac{\mu_1}{\sqrt{2}}$ and $\textnormal{d}(Tx_2,\textnormal{span}(Tx_1))\geq\frac{\mu_2}{\sqrt{2}}$. From the euclidean case we get
		\begin{equation*}
		\underset{t\in\real}{\inf}\|t x_1+(1-t)x_2\|\geq\underset{t\in\real}{\inf}\|t Tx_1+(1-t)Tx_2\|\geq\frac{1}{2\sqrt{5}}\mu_1\mu_2.
		\end{equation*}
	\end{proof}

	\begin{figure}
		\caption{Simplified planar case in \cref{lemmaLine}}\label[figure]{figurePlane}
		\begin{center}
			\includegraphics[width=1\linewidth]{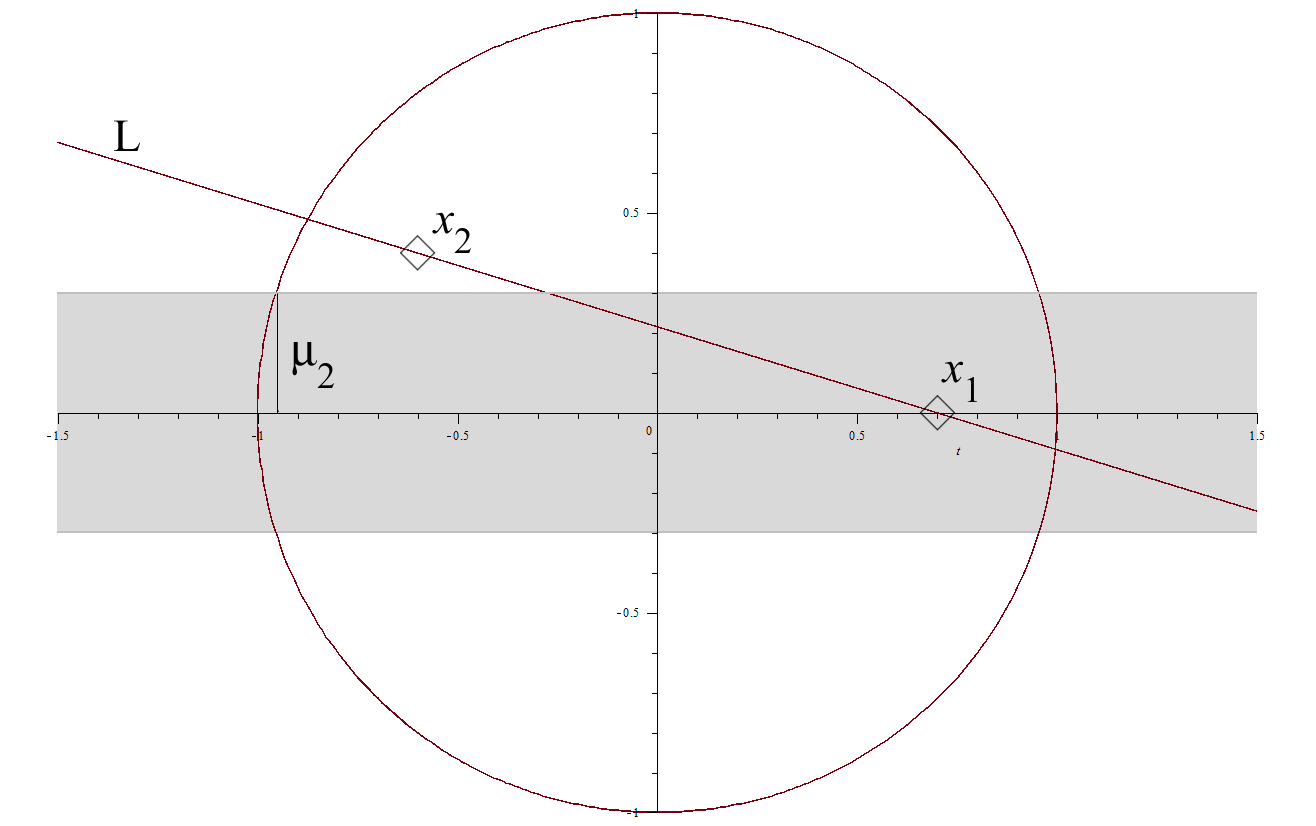}
		\end{center}
	\end{figure}

	\begin{figure}
		\caption{Triangle reduction in \cref{lemmaLine}}\label[figure]{figureTriangle}
		\begin{center}
			\includegraphics[width=1\linewidth]{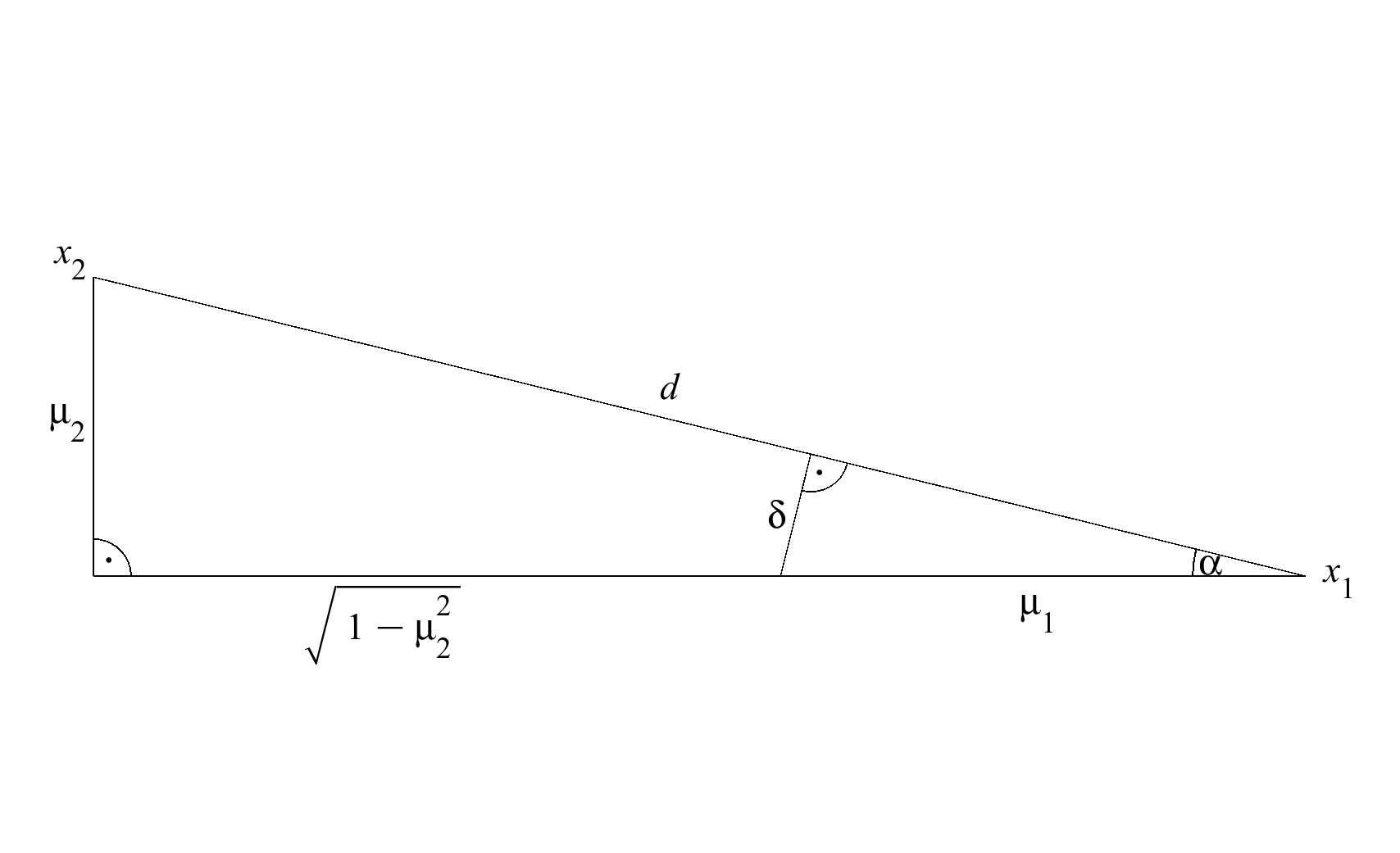}
		\end{center}
	\end{figure}
	
	\begin{proof}[proof of \cref{theoremMainHilbertSpaceExistence} for arbitrary $k$]
		The proof is done by induction over $k$. Assume that the claim holds true for $k\geq 1$. Let $(V_j)_{j\in\natural}\subset\mathcal{G}^{k+1}(X)$ be as in the claim and define $\pi_j:X\to X/V_j$ to be the associated quotient maps. We embed $(V_j)_{j\in\natural}$ into two different sequences of closed, complemented subspaces of $X$, one having codimension $k$ and the other having codimension $1$. Summing their well-separating common complements will yield a well-separating common complement for our initial sequence.\par
		First, take any $({V}^1_j)_{j\in\natural}\subset\mathcal{G}^{k}(X)$ with ${V}^1_j\supset V_j$. According to the codimension $k$ case we find a $\delta^1$-well-separating common complement $C_1\in\mathcal{G}_k(X)$ of $({V}^1_j)_{j\in\natural}$. It holds $\|\pi_j x_1\|\geq\textnormal{d}(x_1,V^1_j)\geq\delta^1_j$ for all $x_1\in C_1$ of norm $1$.\par
		Next, let $V_j^2:=V_j\oplus C_1$. Then, $({V}^2_j)_{j\in\natural}\subset\mathcal{G}^{1}(X)$ is a sequence of closed, complemented subspaces of codimension $1$. Hence, we find a $\delta^2$-well-separating common complement $C_2\in\mathcal{G}_1(X)$. Let $x_2$ be one of the two unit vectors of $C_2$. We have $\textnormal{d}(\pm\pi_j x_2,\pi_j C^1)\geq\delta_j^2$.\par
		Let $C:=C_1\oplus C_2$. To check if $C$ is well-separating, we need to find a lower bound for $\|\pi_j x\|$ with $x\in C$ of norm $1$. We scale $x$ so that it intersects with a boundary element $c$ of the double cone 
		\begin{equation*}
		\Delta:=\left\{c\in C\ \big|\ c=tx_1+(1-t)(\pm x_2),\ t\in[0,1],\ x_1\in B^{C_1}(0,1)\right\},
		\end{equation*}
		which is contained in $B^C(0,1)$ (see \cref{figureCone}). The boundary $\partial\Delta$ is made up of line segments connecting unit vectors $x_1\in C_1$ with one of the two apexes $\pm x_2\in C_2$. By \cref{lemmaLine} the image of each line segment under $\pi_j$ is far enough from the origin, i.e.
		\begin{equation*}
		\underset{t\in[0,1]}{\inf}\|t\pi_jx_1+(1-t)(\pm \pi_jx_2)\|\geq \frac{1}{2\sqrt{5}}\delta_j^1\delta_j^2=:\delta_j.
		\end{equation*}
		Since any $x\in C$ of norm $1$ can be written as $x=\lambda c$ for some $\lambda\geq 1$ and $c\in\partial\Delta$, it holds $\textnormal{d}(x,V_j)=\|\pi_j x\|=\lambda\|\pi_j c\|\geq\lambda\delta_j\geq\delta_j$. Thus, $C$ is a $\delta$-well-separating common complement of $(V_j)_{j\in\natural}$. 
	\end{proof}		

	\begin{figure}
		\caption{Double cone from proof of \cref{theoremMainHilbertSpaceExistence}}\label[figure]{figureCone}
		\begin{center}
			\includegraphics[width=1\linewidth]{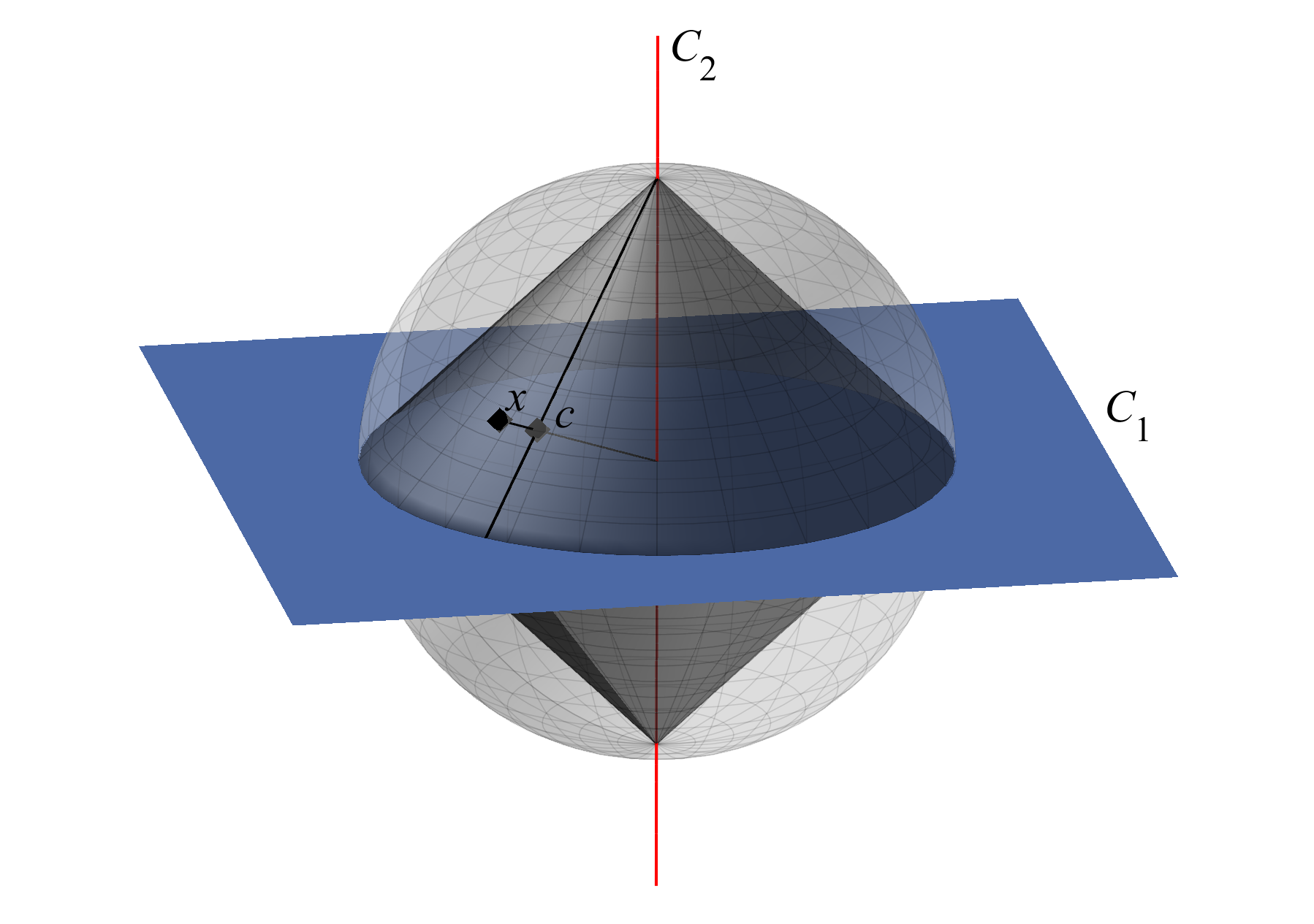}
		\end{center}
	\end{figure}

\subsection{Prevalence}\label{subsectionPrevalence}

	\begin{proposition}\label[proposition]{propositionPrevalenceFiniteDim}
		Let $X$ be a Banach space of dimension $n<\infty$ and let $(V_j)_{j\in\natural}\subset\mathcal{G}^1(X)$ be hyperplanes. Then, almost all $x\in X$ span a well-separating common complement of $(V_j)_{j\in\natural}$.
	\end{proposition}
	\begin{proof}
		Since well-separating common complements are retained when changing to an equivalent norm, we may assume that $(X,\|.\|)=(\real^n,\|.\|_2)$. Furthermore, we can restrict ourselves to $x\in B(0,1)\subset\real^n$. For $\epsilon>0$ define $\delta^{\epsilon}_j:=\epsilon j^{-2}$. We estimate
		\begin{align*}
		\mu(\{x\in B(0,1)\ |\ \exists j:\ \textnormal{d}(x,V_j)\leq\delta^{\epsilon}_j\})&\leq\sum_{j=1}^{\infty}\mu(\{x\in B(0,1)\ |\ \textnormal{d}(x,V_j)\leq\delta^{\epsilon}_j\})\\
		&\leq \sum_{j=1}^{\infty}2\delta^{\epsilon}_j\textnormal{vol}_{n-1}(B^{n-1}(0,1))\\
		&=\epsilon\frac{\pi^2}{3}\textnormal{vol}_{n-1}(B^{n-1}(0,1))\\
		&\underset{\epsilon\to 0}{\longrightarrow}0.
		\end{align*}
		Hence, for almost all $x\in B(0,1)$ there is an $\epsilon>0$ such that $\textnormal{span}(x)$ is a $\delta^{\epsilon}$-well-separating common complement of $(V_j)_{j\in\natural}$.
	\end{proof}
	
	There is no equivalent of the Lebesgue measure for arbitrary Banach or Hilbert spaces. In particular, the proof of \cref{propositionPrevalenceFiniteDim} does not generalize to the infinite-dimensional case. Even the notion of ``a.e.'' in the claim is not clear a priori. Instead of  ``Lebesgue a.e.'' we will use the concept of prevalence \cite{ott2005prevalence}.
	
	\begin{definition}\label[definition]{definitionPrevalence}
		A Borel subset $E\subset X$ of a Banach space is called \emph{prevalent} if there exists a Borel measure $\mu$ on X s.t.
		\begin{enumerate}
			\item $0<\mu(C)<\infty$ for some compact set $C\subset X$, and
			\item $E+x$ has full $\mu$-measure for all $x\in X$.
		\end{enumerate}
		A general subset $F\subset X$ is called \emph{prevalent} if it contains a prevalent Borel set. We say that \emph{almost every} element $x\in X$ lies in $F$.
	\end{definition}
	
	In \cite{ott2005prevalence} it is shown that prevalence satisfies the following genericity axioms.
	
	\begin{proposition}\label[proposition]{propositionPrevalenceAxioms}
		\begin{enumerate}
			The following are true:
			\item $F$ prevalent $\implies F$ dense in $X$,
			\item $L\supset G$, $G$ prevalent $\implies L$ prevalent,
			\item countable intersections of prevalent sets are prevalent,
			\item translations of prevalent sets are prevalent, and
			\item $G\subset\real^n$ is prevalent if and only if $G$ has full Lebesgue measure, i.e. its complement has measure zero.
		\end{enumerate}
	\end{proposition}
	
	The last point implies that the notions of ``a.e.'' in the sense of Lebesgue and in the sense of prevalence coincide in finite-dimensional Banach spaces.\par
	To identify prevalent sets in infinite-dimensional spaces it is convenient to use probe spaces. A \emph{probe} is a finite-dimensional subspace $P\subset X$ of a Banach space. By identification with the standard euclidean space we can equip $P$ with a Borel measure $\lambda_P$. This measure induces a Borel measure $\mu_P$ on $X$ by $\mu_P(A):=\lambda_P(A\cap P)$ for Borel sets $A\subset X$. Using $\mu_P$ in \cref{definitionPrevalence} yields the following definition.
	
	\begin{definition}\label[definition]{definitionProbe}
		A finite-dimensional subspace $P\subset X$ is called a \emph{probe} for $F\subset X$ if there exists a Borel set $E\subset F$ s.t. $E+x$ has full $\mu_P$-measure for all $x\in X$.
	\end{definition} 
	
	\begin{proposition}\label[proposition]{propositionProbe}
		The existence of a probe for $F\subset X$ implies that $F$ is prevalent.
	\end{proposition}
	
	With the additional terminology we are ready to state our main result.
	
	\begin{theorem}\label[theorem]{theoremMainHilbertSpacePrevalence}
		Let $H$ be a Hilbert space and let $(V_j)_{j\in\natural}\subset\mathcal{G}^k(H)$. The set of all $(x_1,\dots,x_k)\in H^k$, such that $\textnormal{span}(x_1,\dots,x_k)$ is a well-separating common complement of $(V_j)_{j\in\natural}$, is prevalent.
	\end{theorem}
	
	\begin{conjecture}\label[conjecture]{conjectureMainBanachSpacePrevalence}
		Let $X$ be a Banach space and let $(V_j)_{j\in\natural}\subset\mathcal{G}^k(X)$. The set of all $(x_1,\dots,x_k)\in X^k$, such that $\textnormal{span}(x_1,\dots,x_k)$ is a well-separating common complement of $(V_j)_{j\in\natural}$, is prevalent.
	\end{conjecture}
	
	We will show that the existence of one well-separating common complement implies prevalence of well-separating common complements. In particular, this proves \cref{theoremMainHilbertSpacePrevalence}. However, before starting with the proof we need a few elementary and technical lemmata.
	
	\begin{lemma}\label[lemma]{lemmaContinuous}
		Let $X$ be a Banach space and $U\subset X$ an open subset. If $f:U\times(\real^k\setminus\{0\})\to\real$ is continuous, then the mapping $g:U\to\real$ defined by 
		\begin{equation*}
			g(x):=\underset{\|\alpha\|_2=1}{\min}f(x,\alpha)
		\end{equation*}
		is continuous as well.
	\end{lemma}
	\begin{proof}
		Let $\epsilon>0$ be given. For each $(x,\alpha)\in U\times(\real^k\setminus\{0\})$, we find $\delta_{(x,\alpha)}>0$ such that
		\begin{equation*}
			\|(x,\alpha)-(y,\beta)\|<\delta_{(x,\alpha)}\implies|f(x,\alpha)-f(y,\beta)|<\epsilon
		\end{equation*}
		for $(y,\beta)\in U\times(\real^k\setminus\{0\})$. Since the set $\{x\}\times\{\|\alpha\|_2=1\}$ is compact, it is covered by finitely many balls of radius $\delta_{(x,\alpha)}$ with $\alpha$ from $\{\|\alpha\|_2=1\}$. Thus, we find $\delta_x>0$ such that
		\begin{equation*}
			\|(x,\alpha)-(y,\beta)\|<\delta_x\implies|f(x,\alpha)-f(y,\beta)|<\epsilon
		\end{equation*}
		for $(y,\beta)\in U\times(\real^k\setminus\{0\})$ if $\|\alpha\|_2=1$. Now, if $\|x-y\|<\delta_x$, then
		\begin{equation*}
			g(x)\leq\underset{\|\alpha\|_2=1}{\min}(f(y,\alpha)+|f(x,\alpha)-f(y,\alpha)|)\leq g(y)+\epsilon.
		\end{equation*}
	\end{proof}
	
	\begin{lemma}\label[lemma]{lemmaBorelSet}
		The set of all tuples spanning well-separating common complements of $(V_j)_{j\in\natural}$ is a Borel subset of $X^k$.
	\end{lemma}
	\begin{proof}
		First, define the map $s:X^k\to\real$ by
		\begin{equation*}
			s(c):=\underset{\|\alpha\|_2=1}{\min}\,\left\|\sum_{i=1}^{k}\alpha_ic_i\right\|.
		\end{equation*}
		With the help of \cref{lemmaContinuous} it is easily seen that $s$ is continuous. In particular, the set $U:=s^{-1}(0,\infty)$ of all linearly independent tuples is open in $X^k$. Next, let $\pi_j:X\to X/V_j$ be the quotient map associated to $V_j$. We apply \cref{lemmaContinuous} again to see that the maps $g_j:U\to\real$ given by 
		\begin{equation*}
			g_j(c):=\underset{\|\alpha\|_2=1}{\min}\frac{\|\sum_{i=1}^{k}\alpha_i\pi_jc_i\|}{\|\sum_{i=1}^{k}\alpha_ic_i\|}
		\end{equation*}
		are continuous. Slightly rewriting $g_j$ reveals that
		\begin{equation*}
			g_j(c)=\underset{x\in\textnormal{span}(c_1,\dots,c_k),\ \|x\|=1}{\inf}\,\textnormal{d}(x,V_j)
		\end{equation*}
		has the form as in \cref{definitionWellSeparatingComplement}. In particular, $\textnormal{span}(c_1,\dots,c_k)$ is a well-separating common complement if and only if $c\in U$, $g_j(c)>0$, and 
		\begin{equation*}
			\lim_{j\to\infty}\frac{1}{j}\log g_j(c)=0.
		\end{equation*}
		Let $f_j:U_j\to\real$ be given by $f_j(c):=j^{-1}\log g_j(c)$, where $U_j:=g_j^{-1}(0,\infty)\subset X^k$ is open. Then, $f_j$ is continuous and bounded from above by zero. Finally, the set of tuples spanning well-separating common complements can be expressed as
		\begin{equation*}
			\bigcap_{l>0}\bigcup_{J\in\natural}\bigcap_{j\geq J}\left\{c\in U_j\ \bigg|\ f_j(c)> -\frac{1}{l}\right\},
		\end{equation*}
		which is a Borel set.
	\end{proof}

	\begin{lemma}\label[lemma]{lemmaDeterminantLowerBoundGeneric}
		Let $(A_j)_{j\in\natural}\subset\real^{k\times k}$ be some matrices. For almost all $A\in\real^{k\times k}$, there exists $\epsilon>0$ s.t.
		\begin{equation*}
			\forall j\in\natural\ :\hspace{1em}|\det(A+A_j)|\geq\frac{\epsilon}{j^2}.
		\end{equation*}
	\end{lemma}
	\begin{proof}
		Let $M>0$ and $\tilde{A}_j:=M^{-1}A_j$. Assume the claim holds for almost all $\tilde{A}\in B(0,1)^k$ w.r.t $(\tilde{A}_j)_{j\in\natural}$. Setting $A:=M\tilde{A}$ for any such $\tilde{A}$ yields
		\begin{equation*}
			|\det(A+A_j)|=M^k|\det(\tilde{A}+\tilde{A}_j)|\geq M^k\frac{\tilde{\epsilon}}{j^2}
		\end{equation*}
		for some $\tilde{\epsilon}>0$. In particular, almost all $A\in B(0,M)^k$ fulfill the required estimate w.r.t. $(A_j)_{j\in\natural}$. Exhausting $\real^{k\times k}$ with $B(0,M)^k$, $M>0$, implies that the claim holds for almost all $A\in\real^{k\times k}$. Thus, it remains to prove that the claim holds for almost all $A\in B(0,1)^k$.\par 
		For $A=(a_1,\dots,a_k)$ it holds
		\begin{align*}
			|\det A|&=|\det (a_1,\Pi_{\textnormal{span}(a_1)^{\perp}}a_2,\dots,\Pi_{\textnormal{span}(a_1,\dots,a_{k-1})^{\perp}}a_k)|\\
			&=\textnormal{vol}_k(\square(a_1,\Pi_{\textnormal{span}(a_1)^{\perp}}a_2,\dots,\Pi_{\textnormal{span}(a_1,\dots,a_{k-1})^{\perp}}a_k))\\
			&=\|a_1\|_2\|\Pi_{\textnormal{span}(a_1)^{\perp}}a_2\|_2\dots\|\Pi_{\textnormal{span}(a_1,\dots,a_{k-1})^{\perp}}a_k\|_2,
		\end{align*}
		where $\Pi_V$ denotes the orthogonal projection onto a subspace $V\subset\real^k$ and $\square(v_1,\dots,v_k)\subset\real^k$ denotes the parallelepiped spanned by vectors $v_1,\dots,v_k\in\real^k$. Using this representation, we will derive an estimate of the form
		\begin{equation}\label{equationMeasureEstimate}
			\mu(\{A\in B(0,1)^k\ :\ |\det(A+\tilde{A})|\leq\eta\}\leq c\eta
		\end{equation}
		for all $\eta>0$ independent of $\tilde{A}$, where $c>0$ is a constant only depending on $k$. To this end fix $\tilde{A}=(\tilde{a}_1,\dots,\tilde{a}_k)$ and define
		\begin{align*}
			t_j(a_1,\dots,a_j)&:=\|a_1+\tilde{a}_1\|_2\|\Pi_{\textnormal{span}(a_1+\tilde{a}_1)^{\perp}}(a_2+\tilde{a}_2)\|_2\\
			&\hspace{2em}\dots\|\Pi_{\textnormal{span}(a_1+\tilde{a_1},\dots,a_{j-1}+\tilde{a}_{j-1})^{\perp}}(a_j+\tilde{a}_j)\|_2
		\end{align*}
		for $j=1,\dots,k$. Set $t_0:=1$. To arrive at an estimate as in \cref{equationMeasureEstimate} we split the integral
		\begin{equation*}
			\int_{B(0,1)^k}\chi_{\{A\ :\ |\det(A+\tilde{A})|\leq\eta\}}(A)\,dA
		\end{equation*}
		using Fubini's theorem. The inner integral becomes
		\begin{equation*}
			I:=\int_{B(0,1)}\chi_{\{a_k\ :\ \|\Pi_{\textnormal{span}(a_1+\tilde{a}_1,\dots,a_{k-1}+\tilde{a}_{k-1})^{\perp}}(a_k+\tilde{a}_k)\|_2\leq\eta t_{k-1}^{-1}\}}(a_k)\,da_k,
		\end{equation*}
		where $t_{k-1}^{-1}$ depends on $a_1,\dots,a_{k-1}$ and might be $\infty$. If it is $\infty$, then the inner integral is $\textnormal{vol}_k(B(0,1))$. In the other case $a_1+\tilde{a}_1,\dots,a_{k-1}+\tilde{a}_{k-1}$ must be linearly independent. Hence, their linear span is of dimension $k-1$ and we find a rotation $T$ that maps $e_1,\dots,e_{k-1}$ into their span and maps $e_k$ into the orthogonal complement. After applying the transformation, we have
		\begin{equation*}
			I=\int_{B(0,1)}\chi_{\{b_k\ :\ \|\Pi_{\textnormal{span}(a_1+\tilde{a}_1,\dots,a_{k-1}+\tilde{a}_{k-1})^{\perp}}(Tb_k+\tilde{a}_k)\|_2\leq\eta t_{k-1}^{-1}\}}(b_k)\,db_k.
		\end{equation*}
		Writing $b_k=(\beta_{1k},\dots,\beta_{kk})^T$ and $\tilde{b}_k=(\tilde{\beta}_{1k},\dots,\tilde{\beta}_{kk})^T$ for $\tilde{b}_k:=T^{-1}\tilde{a}_k$, we get
		\begin{align*}
			I&=\int_{B(0,1)}\chi_{\{b_k\ :\ |\beta_{kk}+\tilde{\beta}_{kk}|\leq\eta t_{k-1}^{-1}\}}(b_k)\,db_k\\
			&\leq 2^{k-1}\int_{-1}^{1}\chi_{\{\beta_{kk}\ :\ |\beta_{kk}+\tilde{\beta}_{kk}|\leq\eta t_{k-1}^{-1}\}}(\beta_{kk})\,d\beta_{kk}\\
			&\leq 2^{k-1}\int_{-1}^{1}\chi_{\{\beta_{kk}\ :\ |\beta_{kk}|\leq\eta t_{k-1}^{-1}\}}(\beta_{kk})\,d\beta_{kk}\\
			&= 2^{k}\min(1,\eta t_{k-1}^{-1})\\
			&\leq 2^{k}\eta t_{k-1}^{-1}.
		\end{align*}
		For the first inequality we embedded $B(0,1)$ into $[-1,1]^k$. Now, we have an estimate on $I$ depending on $a_1,\dots,a_{k-1}$ that also holds when $t_{k-1}^{-1}=\infty$. In the following we show that
		\begin{equation}\label{equationEstimateT}
			\int_{B(0,1)^{k-1}}t_{k-1}^{-1}\,d(a_1,\dots,a_{k-1})\leq c'
		\end{equation}
		for some constant $c'$ by proving that
		\begin{equation}\label{equationInductionMeasure}
			\int_{B(0,1)}t_{k-j}^{-1}\,da_{k-j}\leq c_j't_{k-(j+1)}^{-1}
		\end{equation}
		for some constants $c_j'$ for $j=1,\dots,k-1$. Ultimately, it follows that we can set $c':=c_1'\dots c_{k-1}'$ and $c:=2^kc'$ to reach the desired estimate in \cref{equationMeasureEstimate}. So, let us prove the above inductive formula \cref{equationInductionMeasure}. We write
		\begin{align*}
			&\int_{B(0,1)}t_{k-j}^{-1}\,da_{k-j}\\
			&\hspace{1em}=t_{k-(j+1)}^{-1}\int_{B(0,1)}\|\Pi_{\textnormal{span}(a_1+\tilde{a}_1,\dots,a_{k-(j+1)}+\tilde{a}_{k-(j+1)})^{\perp}}(a_{k-j}+\tilde{a}_{k-j})\|_2^{-1}\,da_{k-j}.
		\end{align*}
		As before, we distinguish between the cases $t_{k-(j+1)}^{-1}=\infty$ and $t_{k-(j+1)}^{-1}<\infty$. In the first case, the inductive formula \cref{equationInductionMeasure} is obviously satisfied. In the second case, we again apply a transformation $T$ which rotates the first $k-(j+1)$ vectors of the standard basis to $\textnormal{span}(a_1+\tilde{a}_1,\dots,a_{k-(j+1)}+\tilde{a}_{k-(j+1)})$ and the remaining basis vectors to its orthogonal complement. Similar to before, writing $b_{k-j}=(\beta_{1(k-j)},\dots,\beta_{k(k-j)})^T$ and $\tilde{b}_{k-j}=(\tilde{\beta}_{1(k-j)},\dots,\tilde{\beta}_{k(k-j)})^T$ for $\tilde{b}_{k-j}:=T^{-1}\tilde{a}_{k-j}$, we get
		\begin{align*}
			&\int_{B(0,1)}\|\Pi_{\textnormal{span}(a_1+\tilde{a}_1,\dots,a_{k-(j+1)}+\tilde{a}_{k-(j+1)})^{\perp}}(a_{k-j}+\tilde{a}_{k-j})\|_2^{-1}\,da_{k-j}\\&\hspace{1em}=\int_{B(0,1)}\|\Pi_{\textnormal{span}(a_1+\tilde{a}_1,\dots,a_{k-(j+1)}+\tilde{a}_{k-(j+1)})^{\perp}}(Tb_{k-j}+\tilde{a}_{k-j})\|_2^{-1}\,db_{k-j}\\
			&\hspace{1em}=\int_{B(0,1)}\|(\beta_{(k-j)(k-j)}+\tilde{\beta}_{(k-j)(k-j)},\dots,\beta_{k(k-j)}+\tilde{\beta}_{k(k-j)})^T\|_2^{-1}\,db_{k-j}.
		\end{align*}
		Let $\beta_{k-j}:=(\beta_{(k-j)(k-j)},\dots,\beta_{k(k-j)})^T$ and $\tilde{\beta}_{k-j}:=(\tilde{\beta}_{(k-j)(k-j)},\dots,\tilde{\beta}_{k(k-j)})^T$. Embedding $B(0,1)\subset\real^k$ into $[-1,1]^{k-(j+1)}\times B(0,1)\subset\real^{k-(j+1)}\times\real^{j+1}$ shows that the above integral can be estimated by
		\begin{align*}
			&2^{k-(j+1)}\int_{B(0,1)}\|\beta_{k-j}+\tilde{\beta}_{k-j}\|_2^{-1}\,d\beta_{k-j}\\
			&\hspace{1em}\leq 2^{k-(j+1)}\int_{B(0,1)}\|\beta_{k-j}\|_2^{-1}\,d\beta_{k-j}=:c_j'<\infty.
		\end{align*}
		Tracing back the steps, this concludes the proof of \cref{equationInductionMeasure}, which in turn gives us \cref{equationEstimateT} and \cref{equationMeasureEstimate}. Having \cref{equationMeasureEstimate}, we set $\eta:=\epsilon j^{-2}$ and $\tilde{A}:=A_j$. It holds
		\begin{align*}
			&\mu(\{A\in B(0,1)^k\ |\ \exists j:\ |\det (A+A_j)|\leq \epsilon j^{-2}\})\\
			&\hspace{1em}\leq\sum_{j=1}^{\infty}\mu(\{A\in B(0,1)^k\ :\ |\det (A+A_j)|\leq \epsilon j^{-2}\})\\
			&\hspace{1em}\leq\sum_{j=1}^{\infty}c\epsilon j^{-2}\\
			&\hspace{1em}=\epsilon \frac{c\pi^2}{6}\\
			&\hspace{1em}\underset{\epsilon\to 0}{\rightarrow}0.
		\end{align*}
		Hence, for almost all $A\in B(0,1)^k$ there is $\epsilon>0$ such that for all $j\in\natural$ we have $|\det(A+A_j)|\geq \epsilon j^{-2}$.
	\end{proof}
	
	\begin{lemma}\label[lemma]{lemmaInverseUpperBoundGeneric}
		Let $(A_j)_{j\in\natural}\subset\real^{k\times k}$ be matrices s.t. $\|A_j\|_2\leq\delta_j^{-1}$ with $0<\delta_j\leq 1$. Then, for almost all $A\in\real^{k\times k}$ there is $\epsilon>0$ s.t.
		\begin{equation*}
			\forall j\in\natural\ :\hspace{1em}\|(A+A_j)^{-1}\|_2^{-1}\geq \epsilon j^{-2}\delta_j^{k-1}.
		\end{equation*}
	\end{lemma}
	\begin{proof}
		Let $A$ be as in \cref{lemmaDeterminantLowerBoundGeneric}. Using the adjugate, we write
		\begin{equation*}
			(A+A_j)^{-1}=\det(A+A_j)^{-1}(A+A_j)^{\textnormal{ad}}.
		\end{equation*}
		Hence, we have
		\begin{equation*}
			\|(A+A_j)^{-1}\|_2^{-1}=|\det(A+A_j)|\,\|(A+A_j)^{\textnormal{ad}}\|_2^{-1}.
		\end{equation*}
		According to \cref{lemmaDeterminantLowerBoundGeneric} the determinant part can be estimated from below by $\tilde{\epsilon}j^{-2}$. For the adjugate part, we remark that the spectral norm and the max norm on $\real^{k\times k}$ are equivalent. Thus, there are constants $c_1,c_2>0$ with $c_1\|.\|_{\max}\leq\|.\|_2\leq c_2\|.\|_{\max}$. Moreover, the entries of the adjugate consist of determinants of $(k-1)\times(k-1)$-matrices with entries from $A+A_j$. As a simple corollary of Hadamard's inequality, we can estimate those determinants using the max norm to obtain
		\begin{align*}
			\|(A+A_j)^{\textnormal{ad}}\|_2&\leq c_2\|(A+A_j)^{\textnormal{ad}}\|_{\max}\\
			&\leq c_2 \|A+A_j\|_{\max}^{k-1}(k-1)^{\frac{k-1}{2}}\\
			&\leq c_2 (k-1)^{\frac{k-1}{2}}c_1^{-(k-1)}\|A+A_j\|_2^{k-1}\\
			&\leq c_2 (k-1)^{\frac{k-1}{2}}c_1^{-(k-1)}(\|A\|_2+\|A_j\|_2)^{k-1}\\
			&\leq c_2 (k-1)^{\frac{k-1}{2}}c_1^{-(k-1)}(\|A\|_2+\delta_j^{-1})^{k-1}\\
			&\leq c_2(k-1)^{\frac{k-1}{2}} c_1^{-(k-1)}(\|A\|_2+1)^{k-1}\delta_j^{-(k-1)}\\
			&=:c\delta_j^{-(k-1)}.
		\end{align*}
		Now, we set $\epsilon:=\tilde{\epsilon}c^{-1}$ to obtain the result.
	\end{proof}

	\begin{proposition}\label[proposition]{propositionPrevalence}
		Assume there exists a well-separating common complement of $(V_j)_{j\in\natural}\subset\mathcal{G}^k(X)$. Then, the set of all $(x_1,\dots,x_k)\in X^k$, such that $\textnormal{span}(x_1,\dots,x_k)$ is a well-separating common complement of $(V_j)_{j\in\natural}$, is prevalent.
	\end{proposition}
	\begin{proof}
		Let $C$ be a $\delta$-well-separating common complement of $(V_j)_{j\in\natural}$. To prove prevalence, we show that the set
		\begin{align*}
			&\big\{(c_1,\dots,c_k)\in C^k\ \big|\ \textnormal{span}(c_1+x_1,\dots,c_k+x_k)\\
			&\hspace{2em}\textnormal{is a well-separating common complement of }(V_j)_{j\in\natural}\big\}
		\end{align*}
		has full Lebesgue measure in the probe space $C^k$ for every translation $(x_1,\dots,x_k)$ of  $X^k$. To get a notion of Lebesgue measure on $C^k$ we identify a basis $b_1,\dots,b_k$ of $C$ with the standard basis $e_1,\dots,e_k$ of $\real^k$. Let us denote this isomorphism by $I:C\to\real^k$. We naturally get an isomorphism $I^k:C^k\to\real^{k\times k}$ mapping elements of $C^k$ to matrices column by column. Thus, we need to check for the measure of all coefficient matrices yielding well-separating common complements. At this point, let us note that the norm on $X^k$ is given by $\|(x_1,\dots,x_k)\|_{X^k}:=\|x_1\|+\dots+\|x_k\|$.\par
		Fix a translation $(x_1,\dots,x_k)$. For each $j\in\natural$ we can write $x_i=c_{ij}'+v_{ij}'$ according to the splitting $X=C\oplus V_j$. The translation contributed by $(c_{1j}',\dots,c_{kj}')$ boils down to a translation on $\real^{k\times k}$ by $A_j:=I^k(c_{1j}',\dots,c_{kj}')$. We are interested in the extend of this translation with increasing $j$. To find an upper bound on the norm of $A_j$, we first assume that $\|c_{ij}'\|>0$. It holds
		\begin{equation*}
			\frac{\|x_i\|}{\|c_{ij}'\|}=\left\|\frac{c_{ij}'}{\|c_{ij}'\|}+\frac{v_{ij}'}{\|c_{ij}'\|}\right\|\geq\textnormal{d}\left(\frac{c_{ij}'}{\|c_{ij}'\|},V_j\right)\geq\delta_j.
		\end{equation*} 
		Thus, we always have $\|c_{ij}'\|\leq\delta_j^{-1}\max_i\|x_i\|$. Switching to the coefficient space, we get $\|A_j\|_2\leq\|I^k\|k\delta_j^{-1}\max_i\|x_i\|$, which can be estimated further by $\tilde{\delta}_j^{-1}:=\max\left(1,\|I^k\|k\delta_j^{-1}\max_i\|x_i\|\right)$.\par 
		Now, let $A$ be as in \cref{lemmaInverseUpperBoundGeneric}. We will show that $A$ induces a well-separating common complement. Let $(c_1,\dots,c_k):=(I^k)^{-1}A$ and let $c\in\textnormal{span}(c_1+x_1,\dots,c_k+x_k)$ with $\|c\|=1$. We can express $c$ in terms of coefficients
		\begin{equation}\label{equationPrevalenceEq1}
			c=\sum_{i=1}^{k}\gamma_i(c_i+x_i)=\sum_{i=1}^{k}\gamma_i\sum_{l=1}^{k}(\alpha_{li}+\alpha^j_{li})b_l+\sum_{i=1}^{k}\gamma_i v_{ij}',
		\end{equation}
		where $A=(\alpha_{il})_{il}$ and $A_j=(\alpha^j_{il})_{il}$. Since $A+A_j$ is invertible by \cref{lemmaInverseUpperBoundGeneric} and $b_1,\dots,b_k$ is a basis, the vectors $\sum_{l=1}^{k}(\alpha_{li}+\alpha^j_{li})b_l$ for $i=1,\dots,k$ form a basis of $C$. In particular, the double sum in \cref{equationPrevalenceEq1} does not vanish. Using the fact that $C$ is $\delta$-well-separating, we compute
		\begin{align*}
			\textnormal{d}(c,V_j)&=\textnormal{d}\left(\sum_{i,l}\gamma_i(\alpha_{li}+\alpha^j_{li})b_l,V_j\right)\\
			&=\Big\|\sum_{i,l}\gamma_i(\alpha_{li}+\alpha^j_{li})b_l\Big\|\,\textnormal{d}\left(\frac{\sum_{i,l}\gamma_i(\alpha_{li}+\alpha^j_{li})b_l}{\left\|\sum_{i,l}\gamma_i(\alpha_{li}+\alpha^j_{li})b_l\right\|},V_j\right)\\
			&\geq \Big\|\sum_{i,l}\gamma_i(\alpha_{li}+\alpha^j_{li})b_l\Big\|\delta_j.
		\end{align*}
		We transfer further norm estimates onto the coefficient space. It holds
		\begin{equation*}
			\Big\|\sum_{i,l}\gamma_i(\alpha_{li}+\alpha^j_{li})e_l\Big\|_2\leq \|I\|\Big\|\sum_{i,l}\gamma_i(\alpha_{li}+\alpha^j_{li})b_l\Big\|.
		\end{equation*}
		Let $\gamma:=(\gamma_1,\dots,\gamma_k)^T$. Using the identity $\gamma=(A+A_j)^{-1}(A+A_j)\gamma$, we get
		\begin{align*}
			\Big\|\sum_{i,l}\gamma_i(\alpha_{li}+\alpha^j_{li})e_l\Big\|_2 &= \|(A+A_j)\gamma\|_2\\
			&\geq \|(A+A_j)^{-1}\|_2^{-1}\|\gamma\|_2\\
			&\geq \epsilon j^{-2}\tilde{\delta}_j^{k-1}\|\gamma\|_2
		\end{align*}
		with $\epsilon>0$ from \cref{lemmaInverseUpperBoundGeneric}. As $(c_i+x_i)_{i=1}^k$ are linearly independent, the norm of $\gamma$ s.t. $\|\sum_{i=1}^{k}\gamma_i(c_i+x_i)\|=1$ for fixed $c_i$ and $x_i$ is bounded from below by a positive constant $\eta>0$.\par 
		Let $\delta_j':=\epsilon j^{-2}\tilde{\delta}_j^{k-1}\eta\|I\|^{-1}\delta_j$. Putting everything together, we have shown that
		\begin{equation*}
			\underset{c\in\textnormal{span}(c_1+x_1,\dots,c_k+x_k),\,\|c\|=1}{\inf}\,\textnormal{d}(c,V_j)\geq\delta_j'
		\end{equation*}
		for all $j\in\natural$, which tells us that $\textnormal{span}(c_1+x_1,\dots,c_k+x_k)$ is a $\delta'$-well-separating common complement of $(V_j)_{j\in\natural}$. Hence, given an arbitrary translation by $(x_1,\dots,x_k)\in X^k$, almost all $A\in\real^{k\times k}$ induce a well-separating common complement.
	\end{proof}

	\begin{remark}\label[remark]{remarkPrevalenceDegree}
		Tracking $\delta'$ in the Hilbert space setting reveals that almost every tuple yields a common complement such that the degree of transversality decays at most polynomially with $\epsilon j^{-(5k^2+2)}$ for some $\epsilon>0$. A better general rate of decay may be obtained by carefully refining the proofs (also see \cref{remarkExistenceHilbertSpaceDelta}).
	\end{remark}

	\bigskip

%% file: Acknowledgments.tex
\section*{Acknowledgments}
	This paper is a contribution to the project M1 (Instabilities across scales and statistical mechanics of multi-scale GFD systems) of the Collaborative Research Centre TRR 181 "Energy Transfer in Atmosphere and Ocean" funded by the Deutsche Forschungsgemeinschaft (DFG, German Research Foundation) - Projektnummer 274762653.

	\bigskip

%% file: backmatter.tex
\bibliographystyle{siam}
\bibliography{sources}